\documentclass[11pt]{article}
\usepackage[a4paper,margin=30mm]{geometry}
\usepackage[T1]{fontenc}
\usepackage{lmodern}
\usepackage{amsmath,amssymb,amsthm,microtype}
\usepackage[hidelinks]{hyperref}
\numberwithin{equation}{section}
\newtheorem{theorem}{Theorem}[section]
\newtheorem{lemma}[theorem]{Lemma}
\newtheorem{corollary}[theorem]{Corollary}
\newtheorem*{theoremA}{Theorem A}
\newtheorem*{theoremB}{Theorem B}
\newtheorem*{theoremC}{Theorem C}

\theoremstyle{remark}
\newtheorem{remark}[theorem]{Remark}
\newtheorem*{problem1}{Problem 1}
\newtheorem*{problem2}{Problem 2}

\newcommand{\N}{\mathbb N}
\title{Minimal asymptotic \(g\)-adic bases from arbitrary partitions}
\author{Min Tang\thanks{Supported by the National Natural Science Foundation of China
(Grant No. 12371003).}\\[3pt]
\small School of Mathematics and Statistics\\
\small Anhui Normal University, Wuhu 241002, P. R. China\\
\small\texttt{tmzzz2000@163.com}}
\date{}
\begin{document}
\maketitle
{\bf Abstract} Let $g>h\ge2$ be integers. For a nonempty set $W\subseteq\N$, let $\mathcal F^*(W)$ denote the set
of its finite nonempty subsets. Let
$$A_g(W)=\left\{\sum_{f\in F}a_fg^f: 1\leq a_f \leq g-1,F \in \mathcal F^*(W)\right\},$$  and let $\N=W_1\sqcup\cdots\sqcup W_h$
be a partition into infinite sets. We prove that
$A=\bigcup_{i=1}^h A_g(W_i)$
is a minimal asymptotic basis of order $h$ for every such partition.


\section{Introduction}
Let $\N=\{0,1,2,\ldots\}$. For $A\subseteq\N$ and integer $h\ge2$, put
\[
hA=\{a_1+\cdots+a_h:a_1,\ldots,a_h\in A\},
\]
\[
\mathcal R_h(n;A)=\left|\{(a_1,\ldots,a_h)\in A^h:
 a_1\le\cdots\le a_h,\ a_1+\cdots+a_h=n\}\right|.
\]

 The set $A$ is an \emph{asymptotic basis of order $h$} if
$\N\setminus hA$ is finite. Such a basis is \emph{minimal} if no proper
subset is an asymptotic basis of order $h$, or equivalently, if
\begin{equation}\label{essential}
E_a:=hA\setminus h(A\setminus\{a\})
\quad\text{is infinite for every }a\in A.
\end{equation}

The counting function of the set $A$ is the function $A(x)$ defined as the number of positive elements of $A$ not exceeding $x$. An asymptotic basis $A$ of order $h$ is \emph{thin} if there is a constant $c>0$ such that
$A(x)<cx^{1/h}$ for all $x$ sufficiently large.
It is \emph{strongly minimal} if, for every $a\in A$, there exists
$c_a>0$ such that $E_a(x)>c_a A(x)^{h-1}$ for all $x$ sufficiently large.

For a nonempty set $W\subseteq\N$, let $\mathcal F^*(W)$ denote the set
of its finite nonempty subsets. For an integer $g\ge2$, let
$$A_g(W)=\left\{\sum_{f\in F}a_fg^f: 1\leq a_f \leq g-1,F \in \mathcal F^*(W)\right\}.$$

The study of minimal asymptotic bases dates back to the 1950s and is closely related to the Erd\H{o}s--Tur\'{a}n conjecture on representation functions. Over the past decade, this topic has attracted renewed interest.

In 1988, Nathanson \cite{Nathanson1988} proved the following result.

\begin{theoremA} Let $h \ge 2$, and let $W_i = \{n \in \mathbb{N} \mid n \equiv i \pmod h\}$ for $i = 0, 1, \dots, h-1$. Then the set
\[
A = A_2(W_0) \cup A_2(W_1) \cup \cdots \cup A_2(W_{h-1})
\]
is a thin, strongly minimal asymptotic basis of order $h$.
\end{theoremA}

In \cite{Nathanson1988}, Nathanson posed the following problem.
\begin{problem1}\label{prob1} {\it Characterize the partition $\mathbb{N} = W_1 \sqcup \cdots \sqcup W_h$ such that $A = \bigcup_{i=1}^h A_2(W_i)$ is a minimal asymptotic basis of order $h$. Is $A$ minimal for ``almost all" partitions?}
\end{problem1}

In 2011, Chen and
Chen \cite{Chen} completely characterized the binary problem for $h=2$. Considerable progress has been made on these problems, see
\cite{ChenSandorYang2024,ChenTang2018}.
A general
characterization is still open. Very recently, Chen and Yan \cite{ChenYan2026} gave a characterization of strongly minimal asymptotic bases of this form in the binary case for every $h\ge2$.

For the $g$-adic problem, Ling and Tang \cite{Ling2015} remarked that if $\mathbb{N} = W_0 \sqcup \cdots \sqcup W_{h-1}$ is a partition such that each part \(W_i\) is infinite and contains \(t\) consecutive integers, where $t$ is the least integer satisfying
$t>\max\left\{1,\log h/\log g\right\}$,
then $A=A_g(W_0)\cup\cdots\cup A_g(W_{h-1})$ is a minimal asymptotic basis of order $h$.
Tang and Ling \cite{TangLing2022} later obtained the following complementary
negative result.

\begin{theoremB}
Let $g\ge2$, and let $h$ and $t$ be integers with $2\le t\le \log h/\log g$.
Then there exists a partition
$\mathbb{N} = W_0 \sqcup \cdots \sqcup W_{h-1}$
such that each set $W_r$ is a union of infinitely many intervals of at least
$t$ consecutive integers and
$A=A_g(W_0)\cup\cdots\cup A_g(W_{h-1})$
is not a minimal asymptotic basis of order $h$.
\end{theoremB}

Combining this theorem with the preceding positive criterion, Tang and Ling showed that if
$\mathbb N=W_1\sqcup\cdots\sqcup W_h$
satisfies the condition that, for every $i$ and every $w\in W_i$, either
$w-1\in W_i$ or $w+1\in W_i$, then the associated $g$-adic set is minimal
for $2\le h<g^2$. For $h\ge g^2$, however, there exist partitions satisfying
the same condition for which minimality fails.

Sun and Cheng \cite{SunCheng2022} proved a complementary existence result.

\begin{theoremC}
Let $t$ be a positive integer, and let $g,h\ge2$. Then there exists a partition
\[
\N=W_0\sqcup W_1\sqcup\cdots\sqcup W_{h-1}
\]
such that each set $W_i$ is a union of infinitely many intervals of $t$
consecutive integers and
\[
A=A_g(W_0)\cup A_g(W_1)\cup\cdots\cup A_g(W_{h-1})
\]
is a minimal asymptotic basis of order $h$.
\end{theoremC}

Taking $t=1$ in Theorem~C shows that, for every $g,h\ge2$, there exists a
partition of $\N$ into $h$ infinite sets for which the associated $g$-adic set
is a minimal asymptotic basis of order $h$.

More recently, Bai and Sun \cite{BaiSun2026} considered the corresponding $g$-adic problem. In this setting it is natural to ask the following analogue of Problem~\ref{prob1}.
\begin{problem2}\label{prob2} {\it Characterize the partition $\mathbb{N} = W_1 \sqcup \cdots \sqcup W_h$ such that $A = \bigcup_{i=1}^h A_g(W_i)$ is a minimal asymptotic basis of order $h$. Is $A$ minimal for ``almost all" partitions?}
\end{problem2}

For related results, see
\cite{ChenTang2026, Fang2026, JiaNathanson1989, LingTang2018, NathansonGadic, SunTao2019}.
In this paper, we show that a much
stronger statement holds when $g>h$: every partition of $\N$ into $h$ infinite
sets has this property, with no further condition on the parts of the
partition.
\begin{theorem}\label{main}
Let $g>h\ge2$ be integers, and let
$$\mathbb{N} = W_1 \sqcup \cdots \sqcup W_h, \quad |W_i|=\infty (1\leq i\leq h).$$
Put $A=\bigcup_{i=1}^{h}A_g(W_i)$.
Then $A$ is a minimal asymptotic basis of order $h$.
Moreover, for every $a\in A$, there are infinitely many integers $n\in E_a$ such that $\mathcal R_h(n;A)=1$.
\end{theorem}

\begin{remark}\label{boundary}
The hypothesis $g>h$ cannot in general be replaced by $g\ge h$.
For $g=h=3$, put
\[
W_1=\{2m:m\in\N\}\cup\{1\},\quad
W_2=\{4m+1:m\ge1\},\quad W_3=\{4m+3:m\in\N\}.
\]
Then
\[
\N=W_1\sqcup W_2\sqcup W_3,\qquad
|W_i|=\infty\quad(1\le i\le3),
\]
and
\[
\{0,1\}\subseteq W_1,\qquad W_2-1\subseteq W_1,\qquad
W_3-1\subseteq W_1,
\]
where $W_i-1=\{w-1:w\in W_i\}$.
By the criterion of Bai and Sun \cite{BaiSun2026},
$\bigcup_{i=1}^{3}A_3(W_i)$ is not a minimal asymptotic basis of order three.
\end{remark}

For $z\in\N$, write
\[
z=\sum_{r\ge0}d_r(z)g^r,\qquad 0\le d_r(z)\le g-1.
\]

\begin{corollary}\label{density}
Let $A$ be as in Theorem~\ref{main}. Let $$U=\left\{n\in\N: \text{for each }1\le j\le h,\text{ there exists }r\in W_j\text{ such that }d_r(n)=g-1
\right\}.$$Then for every $n\in U$ we have $\mathcal R_h(n;A)=1$ and
\[
\lim_{x\to\infty}\frac{U(x)}{x}=1.
\]
\end{corollary}

\begin{corollary}\label{quantitative}
Let $A$ be as in Theorem~\ref{main}. For every $1\le i\le h$ and every
$a\in A_g(W_i)$, there exist constants $c_a>0$ and $m_a\in\N$ such that
\begin{equation}\label{essentialbound}
E_a(g^m)\ge c_a g^{m-W_i(m)}
\end{equation}
for all $m\ge m_a$.
\end{corollary}

\begin{corollary}\label{balanced}
Let $A$ be as in Theorem~\ref{main}, and suppose that $$W_i(m)=\frac{m}{h}+O(1)
\qquad(m\to\infty,\ 1\le i\le h).$$
Then there exist constants $c_1,c_2>0$ such that $$c_1x^{1/h}\le A(x)\le c_2x^{1/h}$$
for all sufficiently large $x$. Moreover, for every $a\in A$, there exists
a constant $c_a>0$ such that $$E_a(x)\ge c_a x^{1-1/h}$$
for all sufficiently large $x$. In particular, $A$ is thin and strongly minimal.
\end{corollary}

\section{Lemmas}

\begin{lemma}[{\cite[Lemma~1]{Lee}}]\label{basis}
Let $g\ge2$ be an integer.
\begin{itemize}
\item[$(a)$] If $W_1,W_2\subseteq\N$ are nonempty and disjoint, then
$A_g(W_1)\cap A_g(W_2)=\varnothing$.
\item[$(b)$] If $W\subseteq\N$ satisfies
$W(x)=\theta x+O(1)$ for some $\theta\in(0,1]$, then there exist constants
$c_1,c_2>0$ such that
\[
c_1x^\theta\le A_g(W)(x)\le c_2x^\theta
\]
for all sufficiently large $x$.
\item[$(c)$] If $h\ge2$ and $\N=W_1\sqcup\cdots\sqcup W_h$ is a
partition into nonempty sets, then $\bigcup_{i=1}^h A_g(W_i)$ is an
asymptotic basis of order $h$.
\end{itemize}
\end{lemma}

\begin{lemma}\label{carry}
Let $g,h\ge2$, and let $x_1,\ldots,x_h\in\N$ with
$n=x_1+\cdots+x_h$. Define
\[
c_0=0,\qquad
c_{r+1}=\left\lfloor
\frac{\sum_{j=1}^{h}d_r(x_j)+c_r}{g}
\right\rfloor\quad(r\ge0).
\]
Then $0\le c_r\le h-1$ and
\begin{equation}\label{carryeq}
\sum_{j=1}^{h}d_r(x_j)+c_r
=d_r(n)+gc_{r+1}\qquad(r\ge0).
\end{equation}
\end{lemma}
\begin{proof}
We first prove $0\le c_r\le h-1$ by induction on $r$.
The assertion holds for $r=0$ since $c_0=0$.
Suppose that $0\le c_r\le h-1$ for some $r\ge0$.
Since $0\le d_r(x_j)\le g-1$ for $1\le j\le h$,
we have
\[
0\le\sum_{j=1}^h d_r(x_j)+c_r
\le h(g-1)+(h-1)=hg-1.
\]
Hence
\[
0\le c_{r+1}
\le\left\lfloor\frac{hg-1}{g}\right\rfloor
=h-1.
\]
Thus the assertion also holds for $r+1$.
By mathematical induction, $0\le c_r\le h-1$ for every $r\ge0$.

For each $r\ge0$, let
\begin{equation}\label{eq2.1}a_r=\sum_{j=1}^{h}d_r(x_j)+c_r-gc_{r+1}.\end{equation}
By the definition of $c_{r+1}$, we have $0\le a_r\le g-1$.

Fix $R\ge0$. We have
\begin{equation}\label{eq2.2}n\equiv\sum_{r=0}^{R}d_r(n)g^r\pmod{g^{R+1}},\qquad
0\le\sum_{r=0}^{R}d_r(n)g^r\le g^{R+1}-1
\end{equation}

On the other hand, by (\ref{eq2.1}) we have
\[
\begin{aligned}
\sum_{r=0}^{R}a_rg^r
&=\sum_{j=1}^{h}\sum_{r=0}^{R}d_r(x_j)g^r
  +\sum_{r=0}^{R}c_rg^r
  -\sum_{r=0}^{R}c_{r+1}g^{r+1}\\
&=\sum_{j=1}^{h}\sum_{r=0}^{R}d_r(x_j)g^r
  -c_{R+1}g^{R+1}.
\end{aligned}
\]
For $j=1,\ldots, h$, we have
$$x_j\equiv\sum_{r=0}^{R}d_r(x_j)g^r \pmod{g^{R+1}},$$
thus \begin{equation}\label{eq2.3}n=\sum_{j=1}^{h}x_j\equiv\sum_{j=1}^{h}\sum_{r=0}^{R}d_r(x_j)g^r\equiv\sum_{r=0}^{R}a_rg^r\pmod{g^{R+1}}.\end{equation}
Noting that $0\le a_r\le g-1$ for $0\le r\le R$, we have
$$0\le\sum_{r=0}^{R}a_rg^r \le(g-1)\sum_{r=0}^{R}g^r=g^{R+1}-1.$$
Hence, by (\ref{eq2.2}) and (\ref{eq2.3}) we have
\[
\sum_{r=0}^{R}a_rg^r=\sum_{r=0}^{R}d_r(n)g^r.
\]
By the uniqueness of $g$-adic representation, we have
$a_r=d_r(n)(0\le r\le R)$.
Since $R$ was arbitrary, we have $a_r=d_r(n)$ for every $r\ge0$.
Substituting this into \eqref{eq2.1} gives \eqref{carryeq}.

\end{proof}
\section{Proof of Theorem~\ref{main}}
By Lemma~\ref{basis}(c), $A$ is an asymptotic basis of order $h$.
Fix $a\in A$. Since $A=\bigcup_{j=1}^{h}A_g(W_j)$, Lemma~\ref{basis}(a) gives a unique
$i\in\{1,\ldots,h\}$ such that $a\in A_g(W_i)$. Put
\[
k=\min\{r\ge0:d_r(a)>0\},\qquad
\ell=\max\{r\ge0:d_r(a)>0\}.
\]
Then $k,\ell\in W_i$, $d_k(a)>0$, and $d_r(a)=0$ for $r<k$ or $r>\ell$.
For each $j\ne i$, choose $t_j\in W_j$ with $t_j>\ell$, and put
\begin{equation}\label{bj}
b_j=(g-1)\left(g^{t_j}+
\sum_{\substack{0\le r<k,r\in W_j}}g^r\right),
\qquad n=a+\sum_{j\ne i}b_j.
\end{equation}
Then $n\in hA$. We shall prove that, up to permutation of the summands, \eqref{bj} is the unique representation of $n$ as a sum of $h$ elements of $A$.

\medskip
\textit{Step 1.} We shall prove
\begin{align}
d_r(n)&=d_r(a) &&(r\in W_i),\label{ndigiti}\\
d_r(n)&=g-1 &&(0\le r<k,\ r\notin W_i),\label{ndigitlow}\\
d_{t_j}(n)&=g-1 &&(j\ne i).\label{ndigitmarker}
\end{align}

Since $a\in A_g(W_i)$, $b_j\in A_g(W_j)$ $(j\ne i)$, and
$W_1,\ldots,W_h$ are pairwise disjoint, for any $r\ge0$ we have
\[
0\le d_r(a)+\sum_{j\ne i}d_r(b_j)\le g-1.
\]
Hence, by \eqref{bj}, for every
$r\ge0$, we have
\begin{equation}\label{ndigits}
d_r(n)=d_r(a)+\sum_{j\ne i}d_r(b_j).
\end{equation}

If $r\in W_i$, then $d_r(b_j)=0$ for every $j\ne i$, and hence
$d_r(n)=d_r(a)$. This proves \eqref{ndigiti}.

If $0\le r<k$ and $r\notin W_i$, then there exists the unique
index $j(\ne i)$ such that $r\in W_j$. By \eqref{bj}, we have
\[
d_r(b_j)=g-1,
\]
while $d_r(a)=0$ and $d_r(b_{j'})=0$ for every
$j'\ne i,j$. Hence $d_r(n)=g-1$, which proves \eqref{ndigitlow}.

Finally, fix $j\ne i$. Since $t_j>\ell\ge k$, we have
$d_{t_j}(b_j)=g-1$.
Since $t_j\in W_j$, we also have
\[
d_{t_j}(a)=0,\qquad
d_{t_j}(b_{j'})=0\quad(j'\ne i,j).
\]
Therefore $d_{t_j}(n)=g-1$, which proves \eqref{ndigitmarker}.

\medskip
\textit{Step 2.}
Consider an arbitrary representation
\begin{equation}\label{arbitrary}
n=x_1+\cdots+x_h,\qquad x_1,\ldots,x_h\in A.
\end{equation}
For $1\le j\le h$, let
\[
q_j=\left|\{p\in\{1,\ldots,h\}:x_p\in A_g(W_j)\}\right|.
\]
Since each $x_p\in A$, Lemma~\ref{basis}(a) implies that each summand
belongs to exactly one $A_g(W_j)$. Hence
\begin{equation}\label{qsum}
q_j\ge0\quad(1\le j\le h),\qquad
\sum_{j=1}^{h}q_j=h.
\end{equation}
We shall prove
\begin{equation}\label{eq3.8}q_j=1\qquad(1\le j\le h).\end{equation}

Let $c_r$ be defined as in Lemma~\ref{carry} for \eqref{arbitrary}.
Then for any $r\ge0$ we have
$0\le c_r\le h-1$ and \eqref{carryeq} holds.

We first prove $q_j\ge1(j\ne i)$ for the arbitrary representation in \eqref{arbitrary}. Fix $j\ne i$ and
suppose that $q_j=0$. By the definition of $q_j$, we have $x_p\notin A_g(W_j)$ for $1\le p\le h$.
Since each $x_p$ belongs to exactly one $A_g(W_s)$ and $t_j\in W_j$,
the pairwise disjointness of $W_1,\ldots,W_h$ gives
\[
d_{t_j}(x_p)=0\qquad(1\le p\le h).
\]
Taking $r=t_j$ in \eqref{carryeq} and using \eqref{ndigitmarker}, we have
$c_{t_j}=g-1+gc_{t_j+1}$.
Hence
\[
h-1\ge c_{t_j}\ge g-1>h-1,
\]
a contradiction. Therefore $q_j\ge1$ for all $j\ne i$.

We next prove that $q_i\ge1$.
Suppose that $q_i=0$. By (\ref{qsum}), $\sum_{j\ne i}q_j=h$.
Since $q_j\ge1$ for every $j\ne i$, it follows that
\begin{equation}\label{qbound}
1\le q_j\le2\qquad(j\ne i).
\end{equation}

 We now prove
\[
c_r=0\qquad(0\le r\le k).
\]
Let $0\le r<k$ and suppose that $c_r=0$.

If $r\in W_i$, then $q_i=0$ implies that for any $1\le p\le h$ we have
$d_r(x_p)=0$.
Moreover, since $r<k$, we have $d_r(a)=0$. By \eqref{ndigiti} we have
\[
d_r(n)=d_r(a)=0.
\]
Hence, by \eqref{carryeq}, we have $gc_{r+1}=0$,
and hence $c_{r+1}=0$.

If $r\in W_j$ for some $j\ne i$, then
\[
d_r(x_p)=0
\qquad\text{for every }p\text{ with }x_p\notin A_g(W_j).
\]
Hence, by \eqref{qbound}, we have
\[
\sum_{p=1}^{h}d_r(x_p)
=\sum_{\substack{1\le p\le h\\x_p\in A_g(W_j)}}d_r(x_p)
\le q_j(g-1)\le2g-2.
\]
Since $d_r(n)=g-1$ by \eqref{ndigitlow}, $c_r=0$, and
\eqref{carryeq}, we have
\[
g-1+gc_{r+1}
=\sum_{p=1}^{h}d_r(x_p)
\le2g-2.
\]
Thus $0\le gc_{r+1}\le g-1$,
and hence $c_{r+1}=0$.

By the above discussion, we know that whenever $0\le r<k$ and $c_r=0$, we have $c_{r+1}=0$.
Since $c_0=0$, it follows successively that
\[
c_r=0\qquad(0\le r\le k).
\]

Since $k\in W_i$ and $q_i=0$, for every $1\le p\le h$ we have
$d_k(x_p)=0$.
Moreover, by \eqref{ndigiti},
\[
d_k(n)=d_k(a)>0.
\]
Using $c_k=0$ in \eqref{carryeq}, we obtain
\[
0
=\sum_{p=1}^{h}d_k(x_p)+c_k
=d_k(n)+gc_{k+1}
=d_k(a)+gc_{k+1}>0,
\]
a contradiction. Therefore
$q_i\ge1$.

In all, for any $1\le j\le h$, we have
$q_j\ge1$. Together with \eqref{qsum}, we obtain
\eqref{eq3.8}.

\medskip

\textit{Step 3.} We shall prove that, after relabelling
$x_1,\ldots,x_h$ in \eqref{arbitrary}, we have
\[
x_i=a,\qquad x_j=b_j\quad(j\ne i).
\]

By \eqref{eq3.8}, after relabelling $x_1,\ldots,x_h$, we may assume that
\begin{equation}\label{eq3.10}x_j\in A_g(W_j)\qquad(1\le j\le h).\end{equation}
For any $r\ge0$, let $j$ be the unique index such that $r\in W_j$.
Then
\[
d_r(x_p)=0\qquad(p\ne j),
\]
and hence
\[
\sum_{p=1}^{h}d_r(x_p)=d_r(x_j)\le g-1.
\]
Since $c_0=0$, an induction on $r$ using the recurrence in Lemma~\ref{carry}
gives $c_r=0$ for every $r\ge0$. Therefore, by \eqref{carryeq},
\begin{equation}\label{uniquedigits}
d_r(n)=d_r(x_j)
\qquad(r\in W_j,\ 1\le j\le h).
\end{equation}

By \eqref{ndigits}, for $r\in W_i$ we have
\[
d_r(n)=d_r(a),
\]
whereas for $r\in W_j$ with $j\ne i$ we have
\[
d_r(n)=d_r(b_j).
\]
It follows from \eqref{uniquedigits} that
\[
d_r(x_i)=d_r(a)\qquad(r\in W_i)
\]
and
\[
d_r(x_j)=d_r(b_j)
\qquad(r\in W_j,\ j\ne i).
\]
By \eqref{eq3.10}, we have $x_i, a\in A_g(W_i)$, thus
\[
d_r(x_i)=d_r(a)=0\qquad(r\notin W_i)
\]
and, for $j\ne i$, noting that $x_j, b_j\in A_g(W_j)$, we have
\[
d_r(x_j)=d_r(b_j)=0\qquad(r\notin W_j).
\]
By the uniqueness of $g$-adic representation, we have
\[
x_i=a,\qquad x_j=b_j\quad(j\ne i).
\]
Hence, $\mathcal R_h(n;A)=1$.
Moreover, the unique representation of $n$ contains $a$. Therefore
\[
n\notin h(A\setminus\{a\}).
\]
Since $n\in hA$ by \eqref{bj}, we obtain
\[
n\in hA\setminus h(A\setminus\{a\})=E_a.
\]

Since $h\ge2$, fix $j_0\ne i$. For each
$j\ne i,j_0$, choose and fix $t_j\in W_j$ with $t_j>\ell$.
For every $t\in W_{j_0}$ with $t>\ell$, put
\[
b_{j_0}(t)
=(g-1)\left(
g^t+\sum_{\substack{0\le r<k,r\in W_{j_0}}}g^r
\right), \quad n(t)
=a+b_{j_0}(t)
+\sum_{\substack{j\not\in\{i,\; j_0\}}}b_j.
\]
By Steps~1--3, for every $t\in W_{j_0}$ with $t>\ell$ we have
\[
n(t)\in E_a,\qquad
\mathcal R_h(n(t);A)=1.
\]

Let $t,t'\in W_{j_0}$ satisfy
$\ell<t<t'$.
Then
$$n(t')-n(t)=b_{j_0}(t')-b_{j_0}(t)=(g-1)(g^{t'}-g^t)>0.$$
Hence the integers $n(t)$ are pairwise distinct.

Since $W_{j_0}$ is infinite and
$\left|\{t\in W_{j_0}:t\le\ell\}\right|<\infty$,
we have
\[
\left|\{t\in W_{j_0}:t>\ell\}\right|=\infty.
\]
Consequently,
\[
\left|\{n\in E_a:\mathcal R_h(n;A)=1\}\right|=\infty.
\]
Since $a\in A$ was arbitrary, $E_a$ is infinite for every $a\in A$.
By \eqref{essential}, $A$ is minimal.

\section{Proofs of the corollaries}

\begin{proof}[Proof of Corollary~\ref{density}]
Fix $n\in U$. By the definition of $U$, for each $1\le j\le h$
there exists $t_j\in W_j$ such that
\begin{equation}\label{densitycondition}
d_{t_j}(n)=g-1.
\end{equation}
Put
\[
b_j=\sum_{r\in W_j}d_r(n)g^r\qquad(1\le j\le h).
\]
By definition, $d_r(b_j)=0$ for $r\notin W_j$.
Since $b_j\ge(g-1)g^{t_j}>0$, we have $b_j\in A_g(W_j)$.
Since $\N=W_1\sqcup\cdots\sqcup W_h$, we obtain
\[
n=b_1+\cdots+b_h.
\]

Consider an arbitrary representation
\[
n=x_1+\cdots+x_h,\qquad x_1,\ldots,x_h\in A,
\]
and put
\[
q_j=\left|\{p\in\{1,\ldots,h\}:x_p\in A_g(W_j)\}\right|
\qquad(1\le j\le h).
\]
Since $A=\bigcup_{j=1}^h A_g(W_j)$, Lemma~\ref{basis}(a) gives
\[
q_j\ge0,\qquad \sum_{j=1}^h q_j=h.
\]
Let $c_r$ be the carries associated with this representation,
as in Lemma~\ref{carry}. If $q_j=0$ for some $1\le j\le h$,
then each $x_p$ belongs to $A_g(W_i)$ for some $i\ne j$.
Since $t_j\in W_j$ and $W_i\cap W_j=\varnothing$, we have
\[
d_{t_j}(x_p)=0\qquad(1\le p\le h).
\]
Hence \eqref{carryeq} and \eqref{densitycondition} give
\[
h-1\ge c_{t_j}=g-1+gc_{t_j+1}\ge g-1>h-1,
\]
a contradiction. Thus $q_j\ge1$ for every $1\leq j\leq h$, and therefore
\[
q_j=1\qquad(1\le j\le h).
\]

After relabelling, we have $x_j\in A_g(W_j)$ for $1\le j\le h$.
For each $r\ge0$, let $j$ be the unique index such that $r\in W_j$.
Since $x_p\in A_g(W_p)$ for $1\le p\le h$, we have
\[
d_r(x_p)=0\qquad(p\ne j),
\]
and hence
\[
\sum_{p=1}^h d_r(x_p)=d_r(x_j)\le g-1.
\]
We now prove that $c_r=0$ for every $r\ge0$.
The assertion holds for $r=0$ since $c_0=0$. If $c_r=0$, then
\[
c_{r+1}
=\left\lfloor
\frac{\sum_{p=1}^h d_r(x_p)+c_r}{g}
\right\rfloor
=\left\lfloor\frac{d_r(x_j)}{g}\right\rfloor=0.
\]
Thus $c_r=0$ for every $r\ge0$ by induction.
By \eqref{carryeq}, we obtain
\[
d_r(n)=d_r(x_j)\qquad(r\in W_j,\ 1\le j\le h).
\]
Since $d_r(x_j)=0$ for $r\notin W_j$, it follows that
\[
x_j
=\sum_{r\in W_j}d_r(x_j)g^r
=\sum_{r\in W_j}d_r(n)g^r
=b_j
\qquad(1\le j\le h).
\]
Hence $\mathcal R_h(n;A)=1$.

\medskip
Next, we estimate the number of positive integers less than $g^m$
that do not belong to $U$.

Fix $m\ge1$. For $n\in\N$ with $1\le n<g^m$, we have $d_r(n)=0$ for every $r\ge m$.
Since $g-1>0$, the definition of $U$ shows that $n\notin U$
if and only if there exists $1\le j\le h$ such that
\[
d_r(n)\ne g-1\qquad(r\in W_j\cap[0,m)).
\]
For $1\le j\le h$, put
\[
B_j(m)=\{n\in\N:1\le n<g^m,\ d_r(n)\ne g-1
\text{ for every }r\in W_j\cap[0,m)\}.
\]
Then
\[
\{n\in\N:1\le n<g^m,\ n\notin U\}
=\bigcup_{j=1}^h B_j(m).
\]

Fix $1\le j\le h$. To count $B_j(m)\cup\{0\}$, write each
$n\in B_j(m)\cup\{0\}$ as
\[
n=d_0(n)+d_1(n)g+\cdots+d_{m-1}(n)g^{m-1}.
\]
If $r\in W_j\cap[0,m)$, then $d_r(n)$ can be any of
$0,\ldots,g-2$, giving $g-1$ choices at each of these
$|W_j\cap[0,m)|$ positions.
If $0\le r<m$ and $r\notin W_j$, then $d_r(n)$ can be any of
$0,\ldots,g-1$, giving $g$ choices at each of the remaining
$m-|W_j\cap[0,m)|$ positions.
By the multiplication principle and the uniqueness of the
$g$-adic representation,
\[
|B_j(m)|+1
=(g-1)^{|W_j\cap[0,m)|}g^{m-|W_j\cap[0,m)|}
=g^m\left(1-\frac1g\right)^{|W_j\cap[0,m)|}.
\]
Consequently,
\begin{equation}\label{exceptionbound}
\begin{aligned}
0&\le(g^m-1)-U(g^m-1)
=\left|\bigcup_{j=1}^h B_j(m)\right|
\le\sum_{j=1}^h|B_j(m)|\\
&=g^m\sum_{j=1}^h
\left(1-\frac1g\right)^{|W_j\cap[0,m)|}-h\le g^m\sum_{j=1}^h
\left(1-\frac1g\right)^{|W_j\cap[0,m)|}.
\end{aligned}
\end{equation}
\medskip
For $x\ge1$, choose $m\ge1$ such that $g^{m-1}\le x<g^m$.
Then
\[
0\le\lfloor x\rfloor-U(x)
\le(g^m-1)-U(g^m-1).
\]
Dividing by $x$ and using \eqref{exceptionbound} and $g^m/x\le g$, we obtain
\[
0\le\frac{\lfloor x\rfloor-U(x)}{x}
\le g\sum_{j=1}^h
\left(1-\frac1g\right)^{|W_j\cap[0,m)|}.
\]
Since every $W_j$ is infinite,
\[
\lim_{m\to\infty}|W_j\cap[0,m)|=\infty
\qquad(1\le j\le h).
\]
Because $h$ is fixed and $0<1-1/g<1$, the right-hand side tends
to zero as $x\to\infty$. Thus
\[
\lim_{x\to\infty}\frac{\lfloor x\rfloor-U(x)}{x}=0.
\]
Together with $\lfloor x\rfloor/x\to1$, this gives
\[
\lim_{x\to\infty}\frac{U(x)}{x}=1.
\]
\end{proof}

\begin{proof}[Proof of Corollary~\ref{quantitative}]
Fix $1\le i\le h$ and $a\in A_g(W_i)$. Choose $k,\ell,t_j,b_j$ as in \eqref{bj}, and choose an integer
\[M>\max_{j\ne i}t_j.\]
For $m\ge M$, put
\[F_m=\{r\in\N:M\le r<m,\ r\notin W_i\}.\]
For each $(\beta_r)_{r\in F_m}\in\{0,\ldots,g-1\}^{F_m}$ and each $j\ne i$, put
\[
b_j'=b_j+\sum_{r\in F_m\cap W_j}\beta_rg^r,
\qquad
n=a+\sum_{j\ne i}b_j'.
\]
Since $M>t_j>\ell$ for $j\ne i$, we have $b_j'\in A_g(W_j)$ and $0<n<g^m$. Moreover,
\[
d_r(n)=d_r(a)\quad(r\in W_i),\qquad
d_r(n)=g-1\quad(0\le r<k,\ r\notin W_i),
\]
and
\[
d_{t_j}(n)=g-1\qquad(j\ne i).
\]
Thus the hypotheses used in Steps~2 and~3 of the proof of Theorem~\ref{main} remain valid, with $b_j$ replaced by $b_j'$. Hence
\[
\mathcal R_h(n;A)=1,\qquad n\in E_a.
\]

For every $r\in F_m$, we have $d_r(n)=\beta_r$. Therefore distinct choices of
$(\beta_r)_{r\in F_m}$ give distinct integers $n$, and hence
\[
E_a(g^m)\ge g^{|F_m|}.
\]
Since $M\ge1$,
\[
|W_i\cap[M,m)|=W_i(m-1)-W_i(M-1),
\]
and therefore
\[
|F_m|=m-M-W_i(m-1)+W_i(M-1)
\ge m-W_i(m)-M+W_i(M-1).
\]
Consequently,
\[
E_a(g^m)
\ge g^{-M+W_i(M-1)}g^{m-W_i(m)}
\qquad(m\ge M).
\]
Thus \eqref{essentialbound} holds with
\[
c_a=g^{-M+W_i(M-1)}>0,\qquad m_a=M.
\]
\end{proof}

\begin{proof}[Proof of Corollary~\ref{balanced}]
By Lemma~\ref{basis}(b), for each $1\le i\le h$ there exist constants $c_{i,1},c_{i,2}>0$ such that
\[c_{i,1}x^{1/h} \le A_g(W_i)(x) \le c_{i,2}x^{1/h}\]
for all sufficiently large $x$. By Lemma~\ref{basis}(a), the sets $A_g(W_1),\ldots,A_g(W_h)$ are pairwise disjoint. Hence
\[A(x)=\sum_{i=1}^{h}A_g(W_i)(x).\]
Since $h$ is fixed, there exist constants $c_1,c_2>0$ such that
\[c_1x^{1/h}\le A(x)\le c_2x^{1/h}\]
for all sufficiently large $x$.

Fix $a\in A$. By Lemma~\ref{basis}(a), there is a unique $i\in\{1,\ldots,h\}$ such that
\[a\in A_g(W_i).\]
By Corollary~\ref{quantitative}, there exist $d_a>0$ and $m_a\in\N$ such that
\[E_a(g^m)\ge d_ag^{m-W_i(m)} \qquad(m\ge m_a).\]
Since
\[W_i(m)=\frac{m}{h}+O(1),\]
there exist $C>0$ and $m_0\in\N$ such that
\[W_i(m)\le\frac{m}{h}+C \qquad(m\ge m_0).\]
Thus, for $m\ge\max\{m_a,m_0\}$,
\[E_a(g^m) \ge d_ag^{-C}g^{m(1-1/h)}.\]

Let $x$ be sufficiently large, and choose $m$ such that
\[g^m\le x<g^{m+1}.\]
By monotonicity of $E_a(x)$, we have
$$E_a(x)\ge E_a(g^m)
\ge d_ag^{-C}g^{m(1-1/h)}
\ge d_ag^{-C-(1-1/h)}x^{1-1/h}.
$$
Hence there exists $c_a'>0$ such that
\[E_a(x)\ge c_a'x^{1-1/h}\]
for all sufficiently large $x$.

The upper bound
$A(x)\le c_2x^{1/h}$
shows that $A$ is thin. Moreover,
\[A(x)^{h-1}\le c_2^{h-1}x^{1-1/h},\]
and therefore
\[E_a(x) \ge \frac{c_a'}{c_2^{h-1}}A(x)^{h-1}\]
for all sufficiently large $x$. Replacing the constant on the right by a smaller positive constant if necessary, $A$ is strongly minimal.
\end{proof}


\begin{thebibliography}{99}
\bibitem{BaiSun2026}
X.L. Bai and C.F. Sun,
\emph{On $3$-adic asymptotic bases of order $h$},
Res. Number Theory \textbf{12} (2026), Paper No. 5.

\bibitem{Chen}
F.J. Chen and Y.G. Chen,
\emph{On minimal asymptotic bases},
European J. Combin. \textbf{32} (2011), 1329--1335.

\bibitem{ChenSandorYang2024}
S.Q. Chen, C. S\'andor and Q.H. Yang,
\emph{On a problem of Nathanson related to minimal asymptotic bases of order $h$},
C. R. Math. Acad. Sci. Paris \textbf{362} (2024), 71--76.

\bibitem{ChenTang2018}
Y.G. Chen and M. Tang,
\emph{On a problem of Nathanson},
Acta Arith. \textbf{185} (2018), 275--280.

\bibitem{ChenTang2026}
C. Chen and M. Tang,
\emph{On problems of Nathanson related to minimal asymptotic bases},
Publ. Math. Debrecen \textbf{109} (2026), 179--191.

\bibitem{ChenYan2026}
S.Q. Chen and X.Y. Yan,
\emph{On strongly minimal asymptotic bases of order $h$ generated by powers of $2$},
arXiv:2610.15416, 2026.

\bibitem{Fang2026}
J.H. Fang,
\emph{On minimal asymptotic bases},
J. Number Theory \textbf{287} (2026), 16--30.

\bibitem{JiaNathanson1989}
X.D. Jia and M.B. Nathanson,
\emph{A simple construction of minimal asymptotic bases},
Acta Arith. \textbf{52} (1989), 95--101.

\bibitem{Lee}
J.B. Lee,
\emph{A construction of minimal asymptotic bases},
Period. Math. Hungar. \textbf{26} (1993), 211--218.

\bibitem{LingTang2018}
D.R. Ling and M. Tang,
\emph{On minimal asymptotic bases of order three},
Colloq. Math. \textbf{151} (2018), 9--18.

\bibitem{Ling2015}
D.R. Ling and M. Tang,
\emph{On minimal asymptotic $g$-adic bases},
Bull. Aust. Math. Soc. \textbf{92} (2015), 374--379.

\bibitem{NathansonGadic}
M.B. Nathanson,
\emph{A new class of minimal asymptotic bases},
in: \emph{Combinatorial and Additive Number Theory V},
Springer Proc. Math. Stat. \textbf{395}, Springer, Cham, 2022,
263--276.

\bibitem{Nathanson1988}
M.B. Nathanson,
\emph{Minimal bases and powers of $2$},
Acta Arith. \textbf{49} (1988), 525--532.

\bibitem{SunCheng2022}
C.F. Sun and Z. Cheng,
\emph{On minimal asymptotic bases},
Math. Notes \textbf{111} (2022), 925--931.

\bibitem{SunTao2019}
C.F. Sun and T.T. Tao,
\emph{On $g$-adic minimal asymptotic bases of order $h$},
Int. J. Number Theory \textbf{15} (2019), 389--406.

\bibitem{TangLing2022}
M. Tang and D. R. Ling,
\emph{On asymptotic bases and minimal asymptotic bases},
Colloq. Math. \textbf{170} (2022), 65--77.
\end{thebibliography}
\end{document}